\documentclass{amsart}
\usepackage[all]{xy}
\usepackage{amssymb}
\usepackage{amsmath}
\usepackage{graphicx}
\usepackage[utf8]{inputenc}
\usepackage[T1]{fontenc}
\usepackage{hyperref}

\newtheorem{thm}{Theorem}[section]
\newtheorem*{thm*}{Theorem}
\newtheorem{lem}[thm]{Lemma}
\newtheorem{prop}[thm]{Proposition}
\newtheorem{cor}[thm]{Corollary}

\newtheorem*{conj*}{Conjecture}
\newtheorem*{assume*}{Assumption}
\theoremstyle{definition}
\newtheorem{defn}[thm]{Definition}
\theoremstyle{remark}
\newtheorem{ex}[thm]{Example}

\newtheorem{iden}{Identity}

\newcommand{\ol}{\overline}
\newcommand{\Z}{\mathbb Z}

\newcommand{\HopfAlg}{\mathit{HopfAlg}}
\newcommand{\CoAlg}{\mathit{CoAlg}}
\newcommand{\Alg}{\mathit{Alg}}
\newcommand{\EHopfAlg}[1]{\mathcal{E}_{#1}\mathit{HopfAlg}}
\newcommand{\SAlg}[1]{\mathcal{S}_{#1}\mathit{Alg}}
\newcommand{\EAlg}[1]{E_{#1}\mathit{Alg}}
\newcommand{\calEAlg}[1]{\mathcal E_{#1}\mathit{Alg}}

\DeclareMathOperator{\Hom}{Hom} 
\DeclareMathOperator{\Sq}{Sq}

\DeclareMathOperator{\coker}{coker}

\title{Brace Bar-Cobar Duality}
\author{Justin Young}
\address{École Polytechnique Fédérale de Lausanne\\
 SB MATHGEOM GR-HE\\ 
MA B3 465 (Bâtiment MA) \\
Station 8 \\
CH-1015 Lausanne}
\email{justin.young@epfl.ch}

\begin{document}

\maketitle
\begin{abstract}Using Kadeishvili's \cite{tKad} formulas with appropriate signs, we show that the classical cobar construction from coalgebras to algebras $\Omega: \CoAlg \to \Alg$ can be enhanced to a functor from Hopf algebras to $E_2$ algebras (for a certain choice of $E_2$ operad) $\Omega : \HopfAlg \to \EAlg{2}$, which, unlike its classical counterpart, is not strictly adjoint, but homotopically equivalent to the left adjoint of the enhanced bar construction $B: \EAlg{2} \to \HopfAlg$ studied by Gerstenhaber--Voronov and Fresse.
\end{abstract}

\tableofcontents
\setcounter{tocdepth}{2}

\section*{Introduction}
The purpose of this article is to prove a precise homotopical relationship between the category of $\mathcal S_2$ algebras and the category of Hopf algebras. By $\mathcal S_2$ we mean the $E_2$ suboperad of the $E_\infty$ operad of McClure--Smith \cite{McSmith2}, called the sequence operad, or the surjection operad. Classical results from \cite{HMS} (Theorems 3.6, 4.4, and 4.5) show that the cobar and bar constructions define an adjunction $\Omega: \CoAlg \rightleftarrows \Alg : B$ between coalgebras and algebras such that the unit and counit maps are homology equivalences. Gerstenhaber--Voronov \cite{GV} (Section 3.2), showed that the bar construction induces a functor $B: \SAlg{2} \to \HopfAlg$. This has been generalized by Fresse \cite{fressehopfbar} (Theorem 5.D, see also \cite{fresse} for specific formulas mod $2$) who showed that the bar construction induces functors $B: \SAlg{n} \to \EHopfAlg{n-1}$, where $\mathcal E_k$ is the Barratt--Eccles $E_k$ operad studied by Berger--Fresse \cite{bergfress}, which has the property that $\mathcal E_1$ is the associative operad (thus the usual bar construction is a special case), and there are equivalences $\mathcal E_k \to \mathcal S_k$ compatible with the respective inclusions $\mathcal E_k \to \mathcal E_{k+1}$ and $\mathcal S_k \to \mathcal S_{k+1}$. In the other direction, Kadeishvili \cite{tKad} showed that the cobar construction induces a functor $\Omega: \HopfAlg \to \SAlg{2}$. This has also been generalized by Fresse in unpublished work showing that the cobar construction induces a functor $\Omega: \EHopfAlg{n-1} \to  \calEAlg{n}$, which for $n=2$ reduces via the map $\mathcal E_2 \to \mathcal S_2$ to Kadeishvili's functor.

From all of this a natural question arises: to what extent can one recover the $E_n$ algebra $A$ from $BA$?

\begin{conj*}[Fresse] $\Omega B A$ and $A$ are equivalent as $\mathcal E_n$ algebras.
\end{conj*}

The conjecture is a classical theorem for $n=1$, and remains open for $n\ge 3$. The main result of this paper is a proof of this conjecture in the case $n=2$. In this case the $\mathcal E_2$ structure reduces via the map $\mathcal E_2 \to \mathcal S_2$ to an $\mathcal S_2$ structure and so the statement is as follows. (This theorem appears in the main text as Corollary \ref{equivthm}.)

\begin{thm*}\label{introthm} If $A$ is an $\mathcal S_2$ algebra, then $\Omega BA$ and $A$ are equivialent as $\mathcal S_2$ algebras.
\end{thm*}

Instead of proving the theorem directly, we study the strict left adjoint of $B$, denoted $\widetilde \Omega: \HopfAlg \to \SAlg{2}$, and prove that there is a natural equivalence of functors $\widetilde \Omega \to \Omega$. As we will show in Section \ref{tildecobar}, the theorem then follows from classical results. The theorem above of theoretical interest, as there does not exist a model category structure on $\HopfAlg$, but according to Mandell \cite{mandellenfpadic}(Section 13, see also Fresse \cite{fressebook} (Part III Section 12)) $\mathcal S_2$ algebras do possess a \emph{semi}-model structure. Thus, in the language of Chachólski--Scherer \cite{JeromeW} (Chapter I Section 5), we have a right semi-model approximation (a generalization of a Quillen equivalence) \[\widetilde \Omega \colon \HopfAlg \rightleftarrows \SAlg{2} \colon B,\] and therefore, there are mapping spaces, homotopy colimits, etc, in the category of Hopf algebras.

The original purpose of this work was to enable the study of $\mathcal S_2$ algebras $A$ in terms of relatively simpler Hopf algebras $BA$. Mandell's Theorem \cite{mandell} (called the Main Theorem in the paper) says that the $\mathcal S$-algebra structure on the cochains $S^*(X)$ of a nice space $X$ is enough to determine the homotopy type of $X$. Since $\mathcal S$ is filtered by suboperads $\mathcal S_n \subset \mathcal S_{n+1} \subset \mathcal S$ one can naturally ask: what homotopy information about $X$ remains if we just consider these ``simpler'' structures? Our eventual goal is to prove that for spaces that are sufficiently connected and relatively low dimensional, $S^*(X)$ is equivalent as an $\mathcal S_2$-algebra to a strictly commutative algebra. The idea is to study instead the Hopf algebra $BS^*(X)$ and show that it is equivalent as a Hopf algebra to a strictly commutative Hopf algebra. This connects to work of Anick \cite{anick}, who showed that for certain $\mathcal S_3$ algebras $A$, $BA$ is equivalent in a weak sense to a commutative Hopf algebra of the form $(UL)^\vee$, the dual of the universal enveloping algebra of a Lie algebra. We prove in the final section that $\Omega (UL)^\vee$ is equivalent as an $\mathcal S_2$ algebra to $C^*(L)$, the Chevalley--Eilenberg cochain complex. This shows that if $BA$ is equivalent as a Hopf algebra to some $(UL)^\vee$, then $A$ is equivalent as an $\mathcal S_2$ algebra to a commutative algebra. We will show in a future paper that for spaces $X$ as above, $BS^*(X)$ is equivalent as a Hopf algebra to $(UL_X)^\vee$, thus $S^*(X)$ is equivalent as an $\mathcal S_2$-algebra to $C^*(L_X)$ which is a strictly commutative algebra.

We now briefly outline the structure of the paper. We enhance the classical bar and cobar constructions to functors $B: \SAlg{2} \to \HopfAlg$, and $\Omega: \HopfAlg \to \SAlg{2}$, respectively, in Sections \ref{barhopf} and \ref{COBAR}. We then study the strict left adjoint of $B$, called the $\mathcal S_2$ cobar construction, $\widetilde \Omega : \HopfAlg \to \SAlg{2}$ and show that there is a natural equivalence of functors $\widetilde \Omega \to \Omega$; this occupies most of Section \ref{tildecobar}. Finally, in Section \ref{chevalley}, we study the $\mathcal S_2$ algebra $\Omega (UL)^\vee$.

I would like to thank Beno\^{i}t Fresse for sharing his work on the cobar construction during a visit in Lille, and for many helpful comments and references. I would also to thank Kathryn Hess for many helpful suggestions on the writing and organization of the paper. Finally, I would like to thank Mike Mandell for providing the original problem and references that gave rise to this work, guiding its development as my advisor, and for many helpful suggestions both mathmatical and stylistic.

\setcounter{section}{-1}

\section{Conventions and Definitions}\label{conv}
In this section we establish conventions and sketch definitions for the notions used in this paper.

Let $R$ denote a commutative ring.

Our convention is that chain complexes $C$ will have differential of degree $-1$. Sometimes, we use ``upper grading" $C^n$ which should be interpreted as notation for $C_{-n}$. We also recall the standard symmetric monoidal structure on the category of chain complexes. In particular, the tensor product $C\otimes D$, and the internal hom $\Hom (C,D)$. We use $C^\vee$ to denote the linear dual chain complex $\Hom (C, R)$. Given a chain complex $C$ we define its suspension $\Sigma C$ to be the chain complex $R[1]\otimes C$, where $R[1]$ is the chain complex given by $R$ concentrated in degree $1$. Similarly we define the desuspension of a chain complex $C$ to be $\Sigma^{-1} C = R[-1]\otimes C$. The signs are determined by the Koszul rule.
We also define a map of chain complexes $f: C\to D$ to be a \emph{weak equivalence} if it induces an isomorphism on homology.

We use $S_*(X, R)$ and $S^*(X,R)$ for simplicial (or singular) chains and cochains, respectively, with coefficients in $R$. If the coefficients are $R = \Z$, we use $S_*(X)$ and $S^*(X)$.

Let $A$ be a chain complex. We say that $A$ is an \emph{associative algebra} provided it is equipped with maps of chain complexes $m: A\otimes A \to A$, $\epsilon: R \to A$, and $\eta : A \to R$ called multiplication, unit and augmentation
respectively. The multiplication map is associative, the unit map satisfies unit conditions, and the augmentation
is a homomorphism of the multiplication and unit maps from $A$ to $R$. For brevity, we will usually use the term \emph{algebra} for associative algebras, and we will specify if we mean some other kind of algebra.

We can write the condition that $m$ is a map of chain complexes element-wise: for $a,b \in A$ we have $d(ab) = d(a)b + (-1)^{|a|}a d(b)$, this is called the \emph{Leibniz Rule}. A map of algebras $f: A\to A'$ is a map of chain complexes that preserves the multiplication, unit and augmentation maps.

There is an obvious dual to the structure of an algebra that is important in this work. Let $C$ be a chain complex. We say that $C$ is a \emph{coassociative coalgebra} provided that it is equipped with maps of chain complexes $\Delta: C\to C\otimes C$, $\epsilon: C\to R$, and $\eta: R\to C$, called comultiplication, counit, and
coaugmentation, respectively. The comultiplication map is required to be coassociative, the counit map satisfies
counit conditions, and the augmentation is a homomorphism of the comultiplication and counit maps from $C$ to $R$.
A map of coalgebras is a map of chain complexes that preserves the comultiplication, counit and coaugmentation maps. Again, we will usually use the term \emph{coalgebra} for coassociative coalgebras.

A coalgebra $C$ is called \emph{conilpotent} if, for all $c\in C$, there exists $N$ such that for $k\ge N$, $\ol \Delta^{k}(c) = 0$. We make the following assumption throughout the paper.
\begin{assume*}All coalgebras are conilpotent.
\end{assume*}

Finally, if we combine an algebra and a coalgebra structure coherently, we obtain the structure of a Hopf algebra. A \emph{Hopf algebra} $H$ is simultaneously a (coassociative) coalgebra and an (associative) algebra such that the multiplication map $m: H\otimes H \to H$ is a map of coalgebras, or equivalently, the comultiplication map $\Delta: H\to H\otimes H$ is a map of algebras. The corresponding commutative diagram is as follows.
\[
\xymatrix{H\otimes H \ar[rr]^m \ar[d]_{\Delta\otimes \Delta} && H \ar[dd]^\Delta\\
H\otimes H\otimes H \otimes H \ar[d]_{1\otimes \tau \otimes 1} && \\
H\otimes H\otimes H\otimes H \ar[rr]_{m\otimes m}&&H\otimes H}
\] where $\tau$ is the twist map.
Note that the counit is an augmentation of the underlying algebra of $H$, and the unit is a coaugmentation of the underlying coalgebra. A map of Hopf algebras is simultaneously a map of algebras and coalgebras.

For an algebra $A$, with augmentation $\epsilon$, we define $\ol A = \ker \epsilon$ and call this the \emph{augmentation ideal}. For a coalgebra $C$, with coaugmentation $\eta$, we define $\ol C = \coker \eta$ and call this the \emph{coaugmentation coideal}. We find that $C$ splits as a chain complex, $C = R \oplus \ol C$. Thus, we can define the reduced diagonal
$\ol \Delta: C\to \ol C\otimes \ol C$ by the composition
\[
\xymatrix{C \ar[r]^\Delta & C\otimes C \ar[r]& \ol C \otimes \ol C }
\]
We will sometimes view $\ol C \otimes \ol C$ as a submodule of $C\otimes C$, which will be clear from the context. We will systematically use sumless notation in this paper. In particular, the ordinary diagonal will be denoted $\Delta (c) = c^{(1)}\otimes c^{(2)}$,
and the reduced diagonal will be denoted $\ol \Delta(c) = \ol c^{(1)} \otimes \ol c^{(2)}$.

A \emph{symmetric sequence of chain complexes} $\mathcal M$ is a collection of chain complexes $\mathcal M(k)$, $k\ge 0$, such that each
$\mathcal M(k)$ is equipped with a $\Sigma_k$ action, where $\Sigma_k$ denotes the symmetric group on $k$ letters. A map of symmetric sequences of chain complexes is a collection of equivariant maps on each component complex. An \emph{operad} is a $\Sigma$ complex $\mathcal P$ together with composition maps \[\gamma: \mathcal P(k)\otimes \mathcal P(i_1)\otimes \cdots \otimes \mathcal P(i_k) \to \mathcal P(i_1 + \cdots + i_k)\]
 \[p\otimes q_1 \otimes \cdots \otimes q_k \mapsto p(q_1, \ldots, q_k)\]and a unit map $\epsilon: R \to \mathcal P(1)$ satisfying appropriate associativity, equivariance, and unit conditions. We also use the notation $p\circ_k q$ for $p(1,\ldots, 1, q, 1, \ldots, 1)$ where $1 = \epsilon(1_R)$, and $q$ is in the $k$th spot.
A map of operads is a map of symmetric sequences of chain complexes that preserves the composition and unit maps. We will insist for the purposes of this work that the unit map is an isomorphism, and also that $\mathcal P(0) = R$.

Let $\mathcal P$ be an operad. A chain complex $A$ is called a $\mathcal P$-algebra if it is equipped with evaluation maps $\{\mathcal P(k) \otimes A^{\otimes k} \to A \mid k\ge 0 \}$ satisfying associativity, equivariance and unit conditions. A map of $\mathcal P$-algebras is a map of chain complexes that preserves the evaluation maps.  A subcomplex $I \subset A$ of a $\mathcal P$-algebra is called an \textit{ideal} if $p(a_1, \ldots, a_k) \in I$ whenever some $a_i \in I$. In this case the quotient $\mathcal P$-algebra is well-defined and denoted by $A/I$.

A chain complex $C$ is called a $\mathcal P$-coalgebra if it is equipped with coevaluation maps $\{\mathcal P(k) \otimes C\to C^{\otimes k} \vert k\ge 0 \}$ satisfying the appropriate dual conditions. A map of $\mathcal P$-coalgebras is a map of chain complexes that preserves the coevaluation maps.

Define the tensor product of symmetric sequences of chain complexes $\mathcal M$ and $\mathcal N$, by $\mathcal M\otimes \mathcal N(k) = \mathcal M(k) \otimes \mathcal N(k)$. If $\mathcal P$ and $\mathcal Q$ are operads, then $\mathcal P \otimes \mathcal Q$ inherits the structure of an operad. An operad $\mathcal H$ is called a \emph{Hopf operad} if it is equipped with a coassociative map of operads $\mathcal H \to \mathcal H \otimes \mathcal H$. If $A$ is an $\mathcal H$-algebra, then $A\otimes A$ inherits the structure of an $\mathcal H$-algebra. If $A$ is also a (coassociative) coalgebra, then we say that $A$ is an \emph{$\mathcal H$ Hopf algebra} provided that the comultiplication $A\to A\otimes A$ is a map of $\mathcal H$ algebras. Similarly, if $C$ is an $\mathcal H$-coalgebra, then $C\otimes C$ inherits the structure of an $\mathcal H$-coalgebra. Suppose that $C$ is also an algebra. We say that $C$ is an \emph{$\mathcal H$ Hopf coalgebra} provided that the multiplication map $C\otimes C \to C$ is a map of $\mathcal H$-coalgebras.

\begin{ex}The associative operad $\mathcal Ass$ given by $\mathcal Ass(k) = R [\Sigma_k]$ is a Hopf operad whose algebras are associative algebras, and whose coalgebras are coassociative coalgebras. Furthermore, an $\mathcal Ass$ Hopf algebra (or $\mathcal Ass$ Hopf coalgebra) is exactly a Hopf algebra.
\end{ex}

\begin{ex}The commutative operad $\mathcal Com$ given by $\mathcal Com(k) = R$ with trivial $\Sigma_k$ action, is a Hopf operad whose algebras are commutative algebras, and whose coalgebras are cocommutative coalgebras.
\end{ex}

Since we insist that all operads $\mathcal P$ satisfy $\mathcal P(0) = R$ and $\mathcal P(1) = R$, we find that there is a canonical map $\mathcal P \to \mathcal Com$ of operads. Therefore, $R$ itself is a $\mathcal P$ algebra and a $\mathcal P$ coalgebra for any operad $\mathcal P$. Thus, it makes sense to speak of augmentations $X\to R$ (coaugmentations $R\to X$) for $\mathcal P$-algebras ($\mathcal P$-coalgebras) $X$. We will assume in general that all $\mathcal P$-algebras are augmented, and all $\mathcal P$-coalgebras are coaugmented. This is equivalent to working nonunitally, however for many purposes it is more convenient to have the unit or counit available. For clarity, we state this assumption separately.
\begin{assume*}All $\mathcal P$-algebras are augmented and all $\mathcal P$-coalgebras are coaugmented. In particular, (co)associative (co)algebras are (co)augmented.
\end{assume*}

Let $\mathcal P$ be an operad, and let $X$ be a chain complex. We define a $\mathcal P$ algebra $\mathbb P X$ as follows. As a complex, $\mathbb P X = \bigoplus \mathcal P(r) \otimes_{\Sigma_r} X^{\otimes r}$, with operad action given by \begin{align*}p(p_1(x_{1,1}, \ldots, x_{1,k_1}), \ldots, p_m(x_{m,1}, \ldots, x_{m, k_m}))=\\  p(p_1,\ldots, p_m)(x_{1,1}, \ldots, x_{1,k_1}, \ldots, x_{m, 1}, \ldots, x_{m, k_m}).\end{align*} We can interpret this as giving $\mathbb P$ the structure of a monad on the category of chain complexes, and the algebras over $\mathbb P$ are exactly the same as the $\mathcal P$ algebras. Note that because of our assumptions on $\mathcal P$, i.e. $\mathcal P(0) = R$, $\mathbb P X$ has a natural unit and augmentation. Note further that $\mathbb P (X)$ is a free $\mathcal P$ algebra on the generating chain complex $X$.

An operad $\mathcal P$ is called an \emph{$E_\infty$ operad} if the map $\mathcal P \to \mathcal Com$ is a weak equivalence, and each $\mathcal P(k)$ is $\Sigma_k$ projective and concentrated in nonnegative degrees. Thus, $\mathcal Com$ itself is not an $E_\infty$ operad, if the ring $R$ does not have enough inverses. An operad $\mathcal P$ is called an \emph{$E_n$ operad} if each $\mathcal P(k)$ is $\Sigma_k$ projective and concentrated in nonnegative degrees, and $\mathcal P$ is weakly equivalent to $S_*(\mathcal C_n)$ as an operad, where $\mathcal C_n$ is the topological operad of little $n$-cubes as defined in \cite{maygils}.

In \cite{McSmith2}, McClure--Smith constructed an $E_\infty$ operad $\mathcal S$ called the sequence operad, and showed
that it acts naturally on simplicial chains (as a coalgebra), and simplicial cochains (as an algebra). McClure--Smith also defined a complexity filtration $\mathcal S_1 \subset \mathcal S_2 \subset \cdots \subset \mathcal S$ such that $\mathcal S_n$ is an $E_n$ operad, and $\mathcal S_1$ is the associative operad. The elements of $\mathcal S(k)_l$ can be represented by formal linear combinations of non-degenerate surjections $f: \{1,\ldots, l+k\} \to \{1, \ldots, k\}$, where non-degenerate means that $f(i) \neq f(i+1)$, for all $i$. We will often write elements of this operad as ordered lists of integers, for $f$ the corresponding list is denoted $(f(1),f(2),\cdots, f(l+k))$. The complexity filtration can be expressed in terms of the number of times any given pair of integers switches order in the list. For example, $(1,2,1,2)$ has complexity $3$ and thus lives in $\mathcal S_3$.

In \cite{bergfress}, Berger--Fresse constructed another $E_\infty$ operad $\mathcal E$ called the Barratt--Eccles operad, and an operad map $\mathcal E \to \mathcal S$, which is therefore a weak equivalence. The operad $\mathcal E$ is given by simplicial chains on a simplicial operad given in arity $k$ by $E\Sigma_k$ the standard contractible $\Sigma_k$-complex. The operad structure comes from the simplicial operad structure on the discrete operad $\mathcal Ass$. Berger--Fresse also define a complexity filtration, coming from a filtration on the simplicial level, $\mathcal E_1 \subset \mathcal E_2 \subset \cdots \subset \mathcal E$ such that the map $\mathcal E \to \mathcal S$ restricts to maps $\mathcal E_n \to \mathcal S_n$, which are weak equivalences of operads. Furthermore, $\mathcal E_1$ is also the associative operad, so the map $\mathcal E_1 \to \mathcal S_1$ is an isomorphism. In addition, since each $\mathcal E_n$ is given by chains on a simplicial operad, each $\mathcal E_n$ is a Hopf operad.

The signs in $\mathcal E$ are thus implicitly determined by the signs in the Eilenberg-Zilber map $S_*(X)\otimes S_*(Y) \to S_*(X\times Y)$, and the signs in the chain differential. Since the morphism $\mathcal E \to \mathcal S$ is surjective, this implicitly determines the operad structure and the signs on $\mathcal S$. We should note that these signs for $\mathcal S$ differ from the signs given by McClure--Smith in \cite{McSmith2}. We use these signs because the signs on $\mathcal E$ are natural, and we would like the $\mathcal E$ and $\mathcal S$ algebra structures to be directly compatible. In \cite{bergfress} (Section 2.2), the authors give explicit formulas for the action of $\mathcal S$, with this choice of signs, on chains (as a coalgebra), and cochains (as an algebra).

Note that we have arranged our conventions so that, for finite-type complexes $X$, $X$ is a $\mathcal P$-algebra if and only if $X^\vee$ is a $\mathcal P$-coalgebra. Similarly, for Hopf operads $\mathcal H$, $X$ is an $\mathcal H$ Hopf algebra if and only if $X^\vee$ is an $\mathcal H$ Hopf coalgebra. We will occasionally switch between these dual points of view without explicitly mentioning it.


\section{The Bar Construction of $\mathcal S_2$ Algebras}\label{barhopf}

The purpose of this section is to define and give first properties of the Hopf algebra structure on $BA$, the bar construction of an $\mathcal S_2$ algebra $A$. The results of this section are due to Gerstenhaber--Voronov \cite{GV}.

Let $A$ denote an algebra. We define the bar construction as follows: $BA = T(\Sigma\ol A) = R \oplus \Sigma \ol A \oplus \Sigma\ol A^{\otimes 2}\oplus \cdots$ with
comultiplication given by the tensor comultiplication, differential
defined by \begin{align*}d([a_1|\cdots |a_k]) &= -\sum_{i=1}^k (-1)^{m_i} [a_1 | \cdots |da_i|\cdots |a_k] \\& +\sum_{i=2}^k (-1)^{m_i} [a_1|\cdots|a_{i-1} a_{i}|\cdots |a_k],\end{align*} where
$m_i = \sum_{j < i} (|a_j| + 1)$, and obvious coaugmentation $R\to BA$. This gives
$BA$ the structure of a coalgebra.

\begin{defn}Let $C$ be a coalgebra and let $A$ be an algebra. An $R$-linear map $f: C\to A$ of degree $-1$ is called a twisting morphism if
$d\circ f + f\circ d = m\circ (f\otimes f) \circ \ol \Delta$
where $m: A\otimes A \to A$ is the multiplication map on $A$ and $\ol \Delta: C\to C\otimes C$ is the reduced diagonal.
\end{defn}

It is a theorem of \cite{HMS} (III, Prop 3.5) that coalgebra maps $C\to BA$ are in one to one correspondence with twisting morphisms $C\to A$. In particular, the
projection $BA\to A$ is the twisting morphism corresponding to the identity $BA\to BA$. Given a twisting morphism $C\to A$ one constructs a morphism of coalgebras $C\to BA$ by iterating diagonals.

For ease of exposition, we recall that an $\mathcal S_2$ algebra is an algebra $A$ over the operad $\mathcal S_2$ with signs as given in \cite{bergfress} (Section 1.2), and briefly outlined in Section \ref{conv}. We note that since the operad $\mathcal S_2$ has known generators and relations, we can describe the structure of an $\mathcal S_2$ algebra as an algebra together with \textit{brace operations} given by \[(1,2,1,\cdots, 1,n,1)(x, y_1, \ldots, y_{n-1}) = x\{y_1, \ldots, y_{n-1}\},\] satisfying appropriate associativity, differential, and mutiplication formulas. We observe that the $\mathcal S_2$ operation $(1,2,3,\cdots, n)$ corresponds to the $n$-fold product: $(1,2,3,\cdots, n)(a_1, \ldots, a_n) = a_1 \cdots a_n$. The relations are determined by computing \[(1,2,1,\cdots,1,n,1) \circ_1 (1,2,1,\cdots, 1,m,1),\] \[d ((1,2,1,\cdots, 1,n,1)),\] and \[(1,2,1,\cdots, 1,n,1) \circ_1 (1,2),\] using the explicit formulas given in \cite{bergfress}. Note that these signs differ from those given in \cite{GV} and in \cite{McSmith2}. We will give the associativity identity without explicit signs because, with the exception of $d((1,2,1)) = (2,1) - (1,2)$, the signs are not critical.
\begin{iden}\label{assoc}
\begin{align*}
x\{x_1,\ldots,x_m\}\{y_1,\ldots, y_n\} &= \sum \pm x\{y_1,\ldots,y_{i_1}, x_{1}\{y_{i_1 + 1},\ldots,y_{j_1}\},y_{j_1 + 1},\ldots\}
\end{align*}
where the sum runs over all sequences $0\le i_1\le j_1\le \ldots i_m\le j_m \le n$.
\end{iden}

\begin{iden}\label{DIFF}
\begin{align*}d(x\{y_1,\ldots,y_n\}) &= (-1)^{n}x\{y_1,\ldots,y_{n-1}\}y_n  \\
&+(-1)^{|x||y_1| + (n-1)|y_1|} y_1 x\{y_2,\ldots,y_n\} \\
&+\sum_{i=2}^n (-1)^{i-1} x\{y_1,\ldots,y_{i-1}y_{i},\ldots,y_n\}\\
&+(-1)^n dx\{y_1,\ldots,y_n\} + \sum_{i=1}^n (-1)^{\gamma_i} x\{y_1,\ldots,dy_i,\ldots,y_n\}.
\end{align*}
where $\gamma_i = n + |x| + \sum_{j=1}^{i-1} |y_j|$.
\end{iden}

\begin{iden}\label{mult}
\begin{align*}(xy)\{y_1,\ldots,y_n\} &= \sum_{i=0}^n (-1)^{\gamma_i} x\{y_1,\ldots,y_i\} y\{y_{i+1},\ldots,y_n\}
\end{align*}
where $\gamma_i = \left(\sum_{k = 1}^i |y_k| \right)|y| + (n-i)\left(|x| + \sum_{j=1}^i |y_j|\right)$.
\end{iden}

Recall also that we require algebras over operads to be augmented, and thus all brace operations vanish if one of the inputs is the unit.

\begin{ex}Suppose that $A$ is a commutative algebra. Define $a\{a_1, \ldots, a_n\} = 0$ for all $a, a_1, \ldots, a_n \in A$. Then, $A$ is an $\mathcal S_2$ algebra.
\end{ex}

We now recall a theorem of \cite{GV} that the structure of an $\mathcal S_2$ algebra on $A$ is equivalent to the structure of a Hopf algebra on $BA$ with a left non-decreasing multiplication (term to be defined below). The structure of a Hopf algebra on the coalgebra $BA$ is determined by a map of differential coalgebras $\cup: BA\otimes BA \to BA$ such that $\cup$ is associative. To say that $\cup$ is left non-decreasing means that the restriction of $\cup$ to $(\bigoplus_{k\ge m}\Sigma\ol A^{\otimes k}) \otimes BA$ lands in $\bigoplus_{k\ge m}\Sigma\ol A^{\otimes k} \subset BA$.

\begin{thm}[Gerstenhaber--Voronov] The structure of an $\mathcal S_2$ algebra on A is equivalent to an associative unital left non-decreasing map of coalgebras $BA\otimes BA\to BA$.
\end{thm}
\begin{proof}The multiplication is uniquely determined by the associated twisting morphism $BA\otimes BA \to A$. By the left non-decreasing condition, this is uniquely determined by the restrictions $E_{n-1}: \Sigma A\otimes {\Sigma A}^{\otimes n-1} \to BA\otimes BA \to A$. In turn, these uniquely determine operations corresponding to the braces \[x\{y_1, \ldots, y_{n-1}\} =  E_{n-1}\circ (\Sigma\otimes \Sigma^{n-1})(x\otimes y_1\otimes \ldots\otimes y_{n-1}).\] One then needs to verify that the twisting morphism formula is equivalent to Identities \ref{DIFF} and \ref{mult} on the associated braces, and that the multiplication is associative if and only if Identity \ref{assoc} holds on the braces.
\end{proof}

\begin{thm}[Gerstenhaber--Voronov] The following is a commutative diagram of functors
\[\xymatrix{\SAlg{2} \ar[r]^B\ar[d] & \HopfAlg \ar[d]\\
\SAlg{1} \ar[r]^B & \CoAlg,}
\]
where the vertical functors are the forgetful ones.
\end{thm}

We will now study $S^*(X, \Z)$ as an $\mathcal S_2$-algebra for what is essentially the smallest nontrivial space: $X=S^1$. We will see that already in this case we have interesting structure. Let $A= H^*(S^1, \Z)$ considered as a commutative $\mathcal S_2$ algebra (with trivial braces). The following example was explained to the author by M. Mandell, and is an unpublished result of Mandell and N. Kuhn.
\begin{ex}[Kuhn--Mandell]The $\mathcal S_2$-algebra $S^*(S^1, \Z)$ is not equivalent as an $\mathcal S_2$ algebra to a commutative algebra, and in particular (in fact, equivalently), $S^*(S^1, \Z)$ is not equivalent to $A$. The Hopf algebra $BS^*(S^1, \Z)$ is given by numerical polynomials $N \subset \mathbb Q [t]$ where $t$ has degree zero, i.e., the polynomials $f(t) \in \mathbb Q[t]$ such that $f(n) \in \Z$ for all $n\in \Z$, which is generated by $\Z$-linear combinations of the following polynomials:
 \[\binom{t}{n} = \frac{t (t - 1) \cdots (t - (n-1))}{n!}. \]
 The Hopf algebra $BA$ is given by the free divided power algebra on one generator $D \subset \mathbb Q[t]$ given by $\Z$-linear combinations of $t^n/n!$.
\end{ex}

The proof of Kuhn--Mandell used sophisticated methods, we will now give a new proof using only straightforward computations. We use the simplicial model $S^1 = \Delta[1]/\partial \Delta[1]$. Let $x \in S^1(S^1, \Z)$ denote the dual of the $1$-simplex given by the quotient map $\Delta[1] \to S^1$. In \cite{bergfress}, Berger--Fresse (see \cite{bergfress} (Fact 3.3.4)) computed explicitly that, for $w\in \mathcal S_2(r)$ we have $w(x, \ldots, x) = 0$ unless $w = 121$ (up to permutations), in which case $121(x, x) = -x$. In brace notation, $x\{x\} = -x$.

Now consider the commutative algebra $A = H^*(S^1)$, as an $\mathcal S_2$ algebra (with trivial braces). The algebra $A$ is the only commutative algebra up to weak equivalence with cohomology $A$, which is easy to see if the ground ring contains $1/2$; a longer argument, which we omit, proves the general case.

We can distinguish $A$ from $S^*(S^1, \Z)$ as $\mathcal S_2$ algebras using extra structure given by the Steenrod operation $\Sq^{n-1}([x])$ where $|x| = -n$, which we will now explain. Let $X$ be an $\mathcal S_2$ algebra over $\Z$. Let $z\in X\otimes \Z/2$ be a cycle with $|z| = -n$, then define $\Sq^{n-1}(z) = z\{z\} \in X\otimes \Z/2$. This defines a function $\Sq^{n-1}: H^n(X\otimes \Z/2)\to H^{2n-1}(X\otimes \Z/2)$, natural in $X$, which will not necessarily be a homomorphism (unless, for example, $X$ is an $\mathcal S_3$ algebra).
If $X$ is commutative, then $Sq^{n-1} = 0$ for all $n$. On the other hand, we have computed $Sq^0([x]) = [x] \neq 0$ in $H^1(S^*(S^1, \Z/2))$. Thus, $S^*(S^1, \Z)$ is not equivalent as an $\mathcal S_2$ algebra to a commutative algebra, or equivalently, it is not formal as an $\mathcal S_2$ algebra.

Now, we compute the Hopf algebra $BS^*(S^1, \Z)$. We introduce the notation $t_n = [x|\cdots| x]$ in the Hopf algebra $BS^*(S^1, \Z)$. One can compute:
\begin{align*}
t_1t_n &= \sum (-1)^\epsilon [x_1|\cdots |x_{i_1}|x_{1}\{x_{i_1 + 1},\ldots ,x_{j_1}\}|x_{j_1 + 1}|\cdots|x_n],
\end{align*}
where the sum is over all sequences $0\le i_1 \le j_1 \le n$, and $x_k = x$ for all $k$. There are two cases in which we get nonzero terms: when $j_1 = i_1 + 1$, in which case we get a $t_n$ term, and when $j_1 = i_1$, in which case we get a $t_{n+1}$ term. By inspecting these cases, we get $n$ terms of the first type and $n+1$ terms of the second type. One can compute that the sign is always $+1$.
One can also compute that $t_nt_1 = t_1t_n$, and so:
\[
t_1 t_n  = t_n t_1 = nt_n + (n+1)t_{n+1}.
\]
Now, if we use the Hopf algebra embedding $BS^*(S^1, \Z) \to BS^*(S^1, \mathbb Q)$, we see that $BS^*(S^1, \mathbb Q)$ is a polynomial algebra on a primitive $t_1$, and that the embedding $BS^*(S^1, \mathbb Z) \to BS^*(S^1, \mathbb Q)$ maps \[t_n \mapsto \frac{t_1 (t_1 - 1) \cdots (t_1 - (n-1))}{n!} = \binom{t_1}{n}.\]
Therefore, we see that $BS^*(S^1, \Z) \cong N$. Here, we use a well-known fact from algebra that the polynomials $\binom{t}{n}$ generate all numerical polynomials over $\Z$.

From this computation it is clear that in $BA$ we have the same linear structure, with generators denoted by $t'_i$. The multiplicative relations are given by \[t'_1 t'_n = t'_n t'_1 = (n+1) t'_{n+1}.\] Thus, via the embedding $A \to A\otimes \mathbb Q$, we find that $BA \cong D$, with $t'_n \mapsto t^n/n!$.

\section{The Cobar Construction of Hopf algebras}\label{COBAR}

Classically, the cobar construction is a functor $\Omega: \CoAlg \to \Alg$, which is left adjoint to the bar construction. In \cite{jAdams} (p. 36), Adams observes that over $\Z /2$ the cobar construction of a Hopf algebra $C$, denoted $\Omega C$, can be equipped with a $\cup_1$ product satisfying
\[
d(x\cup_1 y) = dx \cup_1 y + x \cup_1 dy + xy + yx.
\]
In \cite{tKad}, Kadeishvili proves that this product can be extended to the structure of an $\mathcal S_2$ algebra on $\Omega C$, with $x\{y\} = x\cup_1 y$. The purpose of this section is to work out appropriate signs to show that $\Omega C$, where $C$ is a Hopf algebra, has the structure of an $\mathcal S_2$ algebra over any ground ring $R$. We will then characterize maps $C \to BA$ of Hopf algebras, and in particular observe that the unit map of the adjunction $C\to B\Omega C$ is a map of Hopf algebras.

Let $C$ denote a coalgebra. We define the cobar construction as follows: $\Omega C = T(\Sigma^{-1} \ol C)$ with tensor multiplication, differential
determined by \[d([c]) = -[dc] + (-1)^{|c^{(1)}|} [\ol c^{(1)} | \ol c^{(2)}]\] and the Leibniz Rule, and the obvious augmentation $\Omega C \to R$. This gives $\Omega C$ the
structure of an algebra.

Restricting a morphism $\Omega C \to A$ of algebras $C\to \Omega C \to A$ gives a twisting morphism $C\to A$ (note the degree shift in the inclusion $C\to \Omega C$). This correspondence is a bijection. Thus, the functor $\Omega: \CoAlg \to \Alg$ is adjoint to the functor $B: \Alg \to \CoAlg$.

Observe that if $C$ and $A$ are finite type, then $(\Omega C)^\vee \cong B(C^\vee)$ and $(BA)^\vee \cong \Omega (A^\vee)$.

We now proceed to the goal of this section. We will assume that $C$ is a Hopf algebra, and construct braces such that $\Omega C$ is an $\mathcal S_2$ algebra.
The formula for the first brace is:
\[[x]\{[y_1 | \cdots | y_q]\} = (-1)^\alpha [x^{(1)}y_1 | \cdots | x^{(q)}y_q]\] where $\Delta^{(q)}(x) = x^{(1)}\otimes \cdots \otimes x^{(q)}$ is
the $q$th iterated diagonal, and \[\alpha = \sum_{j=2}^q \left [|x^{(j)}|\sum_{k<j} (|[y_k]|)\right] + |x| - 1.\]

First, we need to verify Identity \ref{DIFF} for the first brace, i.e., we need to verify the following formula.
\begin{lem}
\begin{align*}d([x]\{[y_1 | \cdots | y_q]\}) &= -[x|y_1|\cdots|y_q] + (-1)^{|[x]|(\sum |[y_i]|)}[y_1|\cdots|y_q|x] + \\
&- d[x]\{[y_1|\cdots|y_q]\} + (-1)^{|x|}[x]\{d[y_1|\cdots|y_q]\}.
\end{align*}
\end{lem}
\begin{proof}
We verify the formula in the case that $x, y_1, \ldots, y_q$ are all primitive elements with the zero differential. The general case requires more involved bookkeeping, but poses no additional problems By our assumptions, the formula reduces to the following
\begin{align*}d([x]\{[y_1 | \cdots | y_q]\}) &= -[x|y_1|\cdots|y_q] + (-1)^{|[x]|(\sum |[y_i]|)}[y_1|\cdots|y_q|x].
\end{align*}

By definition of the first brace,
\begin{align*}d([x]\{[y_1 | \cdots | y_q]\}) &= (-1)^\alpha \sum_{i=1}^q (-1)^{m_i}[x^{(1)}y_1 | \cdots| d([x^{(i)}y_i])| \cdots | x^{(q)}y_q],\\
\end{align*}
where $m_i$ is determined by the Koszul rule.

Since $x$ is primitive,
\[[x]\{[y_1 | \cdots | y_q]\}= (-1)^{|x|-1} \sum_{i=1}^q (-1)^{\alpha_i} [y_1| \cdots | xy_i | \cdots | y_n], \]
where $\alpha_i = |x|\sum_{k<i} (|y_k|-1)$. Thus, since $\ol \Delta (xy_i) = x\otimes y_i + (-1)^{|x||y_i|} y_i \otimes x$, we have that
\begin{align*}d([x]\{[y_1 | \cdots | y_q]\}) &= (-1)^{|x| - 1} \sum_{i=1}^q (-1)^{\beta_i}[y_1| \cdots | x | y_i | \cdots | y_n] +
\\&(-1)^{|x| - 1} \sum_{i=1}^q (-1)^{\gamma_i}[y_1| \cdots | y_i | x | \cdots | y_n],
 \end{align*}
 where \[\beta_i = \alpha_i + \sum_{k<i} (|y_k| - 1) + |x| \] and \[\gamma_i = \alpha_i + \sum_{k<i} (|y_k| - 1) + |x||y_i| + |y_i|.\]
Thus,
\begin{align*}\gamma_{i-1} &= \alpha_{i-1} + \sum_{k<i-1} (|y_k| - 1) +  |x||y_{i-1}| + |y_{i-1}| \\
&= \beta_i + |x||y_{i-1}| + |x| + |y_{i-1}| - 1 + |x| +  |x||y_{i-1}| + |y_{i-1}|\\
&= \beta_i - 1.
\end{align*}
So, we have a telescoping sum with remaining terms as follows,
\[(-1)^{\beta_1 + |x| - 1}[ x | y_1 | \cdots | y_n] + (-1)^{\gamma_n + |x| - 1}[y_1| \cdots | y_n | x],\]
and it is easy to verify that the signs are correct.
\end{proof}

We define $[x_1|\cdots|x_n]\{[y_1|\cdots|y_q]\}$ using Identity \ref{mult}. Identity \ref{mult} is consistent with associativity of the multiplication, i.e., \[(1,2,1,\cdots, 1,n,1) \circ_1 ((1,2)\circ_1 (1,2)) = (1,2,1,\cdots, 1,n,1) \circ_1 ((1,2)\circ_2 (1,2)),\] and so Identity \ref{mult} will necessarily be satisfied for the first brace. Similarly, Identity \ref{mult} is consistent with Identity \ref{DIFF}, i.e., \[d((1,2,1,\cdots, 1,n,1) \circ_1 (1,2)) = d((1,2,1,\cdots, 1,n,1))\circ_1 (1,2)\] and so by induction Identity \ref{DIFF} is also satisfied in general for the first brace.

Next, we need to define the higher braces. We start with the initial condition $[x]\{\beta_1,\ldots,\beta_p\} = 0$ for $p > 1$.
Then, we extend using Identity \ref{mult}. Now, the following expression,
\begin{align*} [x]\{\beta_1\}\beta_2 + (-1)^{|[x]||\beta_1| + |\beta_1|}\beta_1 ([x]\{\beta_2\})\\
- [x]\{\beta_1 \beta_2\} + d([x])\{\beta_1, \beta_2\}
\end{align*}
is zero because of the definition \[[x_1|x_2]\{\beta_1, \beta_2\} = (-1)^{|[x_2]||\beta_1| + |[x_1]| + |\beta_1|} [x_1]\{\beta_1\}[x_2]\{\beta_2\}.\] Thus, our initial condition is consistent with Identity \ref{DIFF} in the case $p = 2$. The corresponding formulas for $d([x]\{\beta_1,\ldots, \beta_p\})$ for $p > 2$ are completely trivial. As above, by consistency, it follows that Identities \ref{DIFF} and \ref{mult} hold in general for all braces.

Finally, we need to look at Identity \ref{assoc}. Consider the case \[[x]\{[y_1|\cdots|y_l]\}\{\beta_1,\ldots,\beta_n\}\] where $l\ge n$. Since the required operations mostly act by zero, we only need to look at one sign. One can compute that
\begin{align*}(1,2,1,\cdots, 1,(n+1),1)\circ_1 (1,2,1) &= \\(-1)^n (1,2,1)\circ_2 (1,2,1,\cdots, 1,(n+1),1) +
&\text{ terms acting by zero.}
\end{align*}

\begin{lem}\[[x]\{[y_1|\cdots|y_l] \}\{\beta_1,\ldots,\beta_n\} = (-1)^{n+ n|[x]|} [x]\{[y_1|\cdots|y_l]\{\beta_1, \ldots, \beta_n\}\}\]
\end{lem}
\begin{proof}By definition, we have $[x]\{[y_1|\cdots|y_l]\} = (-1)^\alpha [x^{(1)}y_1 |\cdots |x^{(l)}y_l]$. Thus, \[[x]\{[y_1|\cdots|y_l] \}\{\beta_1,\ldots,\beta_n\} =\]
\[(-1)^\alpha \sum_{i=1}^{l-n} (-1)^{\delta_i}[x^{(1)}y_1|\ldots][x^{(i)}y_i]\{\beta_1\}\cdots [x^{(i+n)}y_{i+n}]\{\beta_n\}\}[\cdots|x^{(l)}a_l] \]
On the other hand,
\[
(-1)^{n+ n|[x]|} [x]\{[y_1|\cdots|y_l]\{\beta_1, \ldots, \beta_n\}\}=
\]
\[
(-1)^{n+n|[x]|}\sum_{i=1}^{l-n} (-1)^{\delta_i}[x]\{[y_1|\cdots][y_i]\{\beta_1\}\cdots [y_{i+n}]\{\beta_n\}[\cdots|y_l]\}
\]
As a consequence of our definitions, we see that we get the same terms, and a trivial but tedious check shows that the signs work out.
\end{proof}
Thus, we have a proof of Identity \ref{assoc} in this case. Since Identity \ref{assoc} is consistent with Identity \ref{mult}, by induction, Identity \ref{assoc} holds in general.

This concludes the proof
that the braces defined above satisfy Identities \ref{assoc}, \ref{DIFF} and \ref{mult}.

We have now established that if $C$ is a Hopf algebra then $\Omega C$ is an $\mathcal S_2$ algebra, and we clearly have a functor $\Omega: \HopfAlg \to \SAlg{2}$. We would now like to study the universal properties of $\Omega$ and $B$.

Suppose that $C$ is a Hopf algebra, and let $C\to A$ be a twisting morphism. To say that the associated coalgebra map $C\to BA$ is a map of Hopf algebras boils
down to the commutativity of the diagram:
\[
\xymatrix{C\otimes C\ar[r]\ar[d]&BA\otimes BA \ar[d]\\
C\ar[r]& BA.}
\]
Thus, we need to compare the two twisting morphisms $C\otimes C \to A$ coming from the two ways of following the diagram $C\otimes C \to BA$ followed by
the universal twisting morphism $BA\to A$. Consider $c_1, c_2 \in C$ and let \[\ol\Delta^{(k)}(c_2) = \ol c_2^{(1)}\otimes \cdots\otimes \ol c_2^{(k)}.\] If we let $f: C\to A$ denote the original twisting morphism, going around the top leads to the expression \[\sum (-1)^\alpha f(c_1)\{f(\ol c_2^{(1)}),\ldots, f(\ol c_2^{(k)})\}\] where $\alpha = k(|c_1|-1) + \sum_{i=1}^k (k-i)(|\ol c_2^{(i)}| - 1)$ and going around the bottom leads to $f(c_1c_2)$. This leads to the following characterization of Hopf algebra maps $C\to BA$.

\begin{defn}\label{hopftwist}Let $C$ be a Hopf algebra and let $A$ be an $\mathcal S_2$ algebra. A twisting morphism $f: C\to A$ is called \emph{Hopf} if
\[f(c_1c_2)=\sum (-1)^\alpha f(c_1)\{f(\ol c_2^{(1)}),\ldots, f(\ol c_2^{(k)})\}, \] for all $c_1, c_2 \in \ol C$, where $\alpha = k(|c_1|-1) + \sum_{i=1}^k (k-i)(|\ol c_2^{(i)}| - 1).$
\end{defn}

\begin{prop}The canonical bijection between twisting morphisms $C\to A$ and morphisms of coalgebras $C\to BA$ restricts to a bijection between Hopf twisting morphisms $C\to A$ and morphisms of Hopf algebras $C\to BA$.
\end{prop}

We now establish a technical theorem about Hopf twisting morphisms that is useful for constructing such morphisms. This will play a key role in Section \ref{tildecobar}.

\begin{thm}\label{techtwist}Let $C$ be a Hopf algebra and let $A$ be an $\mathcal S_2$ algebra. Let $C' \subset C$ be a subcoalgebra with the property that if $c_1, c_2 \in C'$, then $d(c_1)c_2, c_1 d(c_2) \in C'$, and $\ol c_1^{(1)} \ol c_2^{(1)}, \ol c_1^{(2)} \ol c_2^{(2)}, \ol c_1^{(1)}\ol c_2^{(2)} \in C'$ . Let $f: C' \to A$ be a twisting morphism that satisfies the Hopf twisting morphism formula (see Definition \ref{hopftwist}) on elements of the form $d(c_1)c_2, c_1 d(c_2), \ol c_1^{(1)} \ol c_2^{(1)}, \ol c_1^{(2)} \ol c_2^{(2)}, \ol c_1^{(1)}\ol c_2^{(2)} $ for $c_1, c_2 \in C'$. Let $x, y \in C'$, and define \[\Gamma = \sum (-1)^\alpha f(x)\{f(\ol y^{(1)}),\ldots, f(\ol y^{(k)})\},\] where $\alpha = k(|x|-1) + \sum_{i=1}^k (k-i)(|\ol y^{(i)}| - 1)$. Then, \[d\Gamma + f(d(xy)) = m\circ (f\otimes f) \circ \ol \Delta (xy).\]
\end{thm}
\begin{proof}The following composition induced by $f$
\[U: C'\otimes C' \to BA\otimes BA \to BA \to A, \] is a twisting morphism.
Let $D \subset C'\otimes C'$ be the submodule generated by the images of $d$ and $\ol \Delta$. Then, $D$ is a subcoalgebra, and the product
defines a map $p: D \to C'$. Now, we can use this map together with $f$ to induce a twisting morphism
\[L: D \to C' \to A.\] We see by assumption that on elements of $D$, $U = L$. Thus, we have the following equations:
\[U(d(xy)) = f(d(xy)),\]
\[m\circ (U\otimes U) \circ \ol \Delta (x\otimes y) = m\circ (f\otimes f)\circ \ol \Delta(xy),\]
and
\[U(x\otimes y) = \Gamma. \]
The conclusion follows from the fact that $U$ is a twisting morphism.
\end{proof}

\begin{cor}\label{hopftwistdef}Let $C$ be a connected nonnegatively graded Hopf algebra and let $A$ be an $\mathcal S_2$ algebra. Let $c_1, c_2 \in \ol C$. Let $f: C\to A$ be defined on all elements $c$ satisfying $|c| < |c_1| + |c_2|$ so that the twisting morphism formula $df(c) + fd(c) = m(f\otimes f)\ol\Delta(c)$ holds, and if $x_1, x_2\in \ol C$ with $|x_1| + |x_2| < |c_1| + |c_2|$, then \[f(x_1x_2)=\sum (-1)^\alpha f(x_1)\{f(\ol x_2^{(1)}),\ldots, f(\ol x_2^{(k)})\}. \] If we set \[f(c_1 c_2) = \sum (-1)^\alpha f(c_1)\{f(\ol c_2^{(1)}),\ldots, f(\ol c_2^{(k)})\},\] then the twisting morphism formula $df(c_1 c_2) + f d(c_1 c_2) = m(f\otimes f)\ol\Delta(c_1 c_2)$ holds.
\end{cor}
\begin{proof} Let $C'\subset C$ be the subcoalgebra given by elements in degrees less than $|c_1| + |c_2|$, and apply Theorem \ref{techtwist}. \end{proof}
As a direct corollary of this we have the following useful result.
\begin{cor}Let $C = TV$ be a Hopf algebra whose algebra structure is free on the generators $V$, where $V$ is non-negatively graded and $V_k = 0$ for $k\ge n+1$. Let $C'= TV_{\le n-1} \subset C$ be the sub-Hopf algebra given by the generators in degrees $\le n-1$. Assume that $C'\to A$ is a Hopf twisting morphism, and there is an extension to \[V \oplus_{V_{\le n-1}} TV_{\le n-1} \to A\] that is a twisting morphism. Then, the unique extension $C \to A$ given by the formula of Lemma \ref{hopftwistdef} is a Hopf twisting morphism.
\end{cor}
Observe that the extension is unique by Identity \ref{assoc}, comparing terms, and sign calculations.
The importance of this corollary is that it allows one to construct Hopf twisting morphisms with source a free Hopf algebra (of the form $C = TV$ as an algebra) inductively on the generators $V$. For example, the Hopf algebra could be of the form $\Omega S_*(X)$, this example will be used in a future paper.

In a different direction, we have the following variation.

\begin{cor}\label{freehopf} Let $V$ be a free graded $R$-module with the zero differential, and define $C = TV$ to be the Hopf algebra with $V$ as its primitive elements. Let $A$ be an $\mathcal S_2$ algebra. Then, any degree $-1$ $R$-linear map $f: V \to A$ extends uniquely to a Hopf twisting morphism $f: TV \to A$.
\end{cor}
\begin{proof} By induction on length, Theorem \ref{techtwist} implies that the unique extension $f: TV \to A$ is still a twisting morphism.
\end{proof}

If $A = \Omega C$ then $f(c) = [c]$, and by definition above, we have
\begin{align*}\sum (-1)^\alpha [c_1]\{[\ol c_2^{(1)}],\ldots, [\ol c_2^{(k)}]\} &= (-1)^{|c_1| -1}[c_1]\{[c_2]\}\\ &= [c_1c_2]
\end{align*}
and so the universal twisting morphism $C\to \Omega C$ is a Hopf twisting morphism and thus induces a map of Hopf algebras $C\to B\Omega C$.

We now have the following generalization of the main result of \cite{tKad} to any ground ring $R$.
\begin{thm}\label{unit}The cobar construction induces a functor $\Omega: \HopfAlg \to \SAlg{2}$.
Furthermore, the natural weak equivalence $C \to B\Omega C$ is a map of Hopf algebras.
\end{thm}

In \cite{hess}, Hess--Parent--Scott--Tonks study $\Omega S_*(K)$, where $K$ is a $1$-reduced simplicial set, as a coalgebra, and its relationship to $S_*(GK)$, where $GK$ is the Kan loop group on $K$. Note that the product $GK \times GK \to GK$ induces a Hopf algebra structure on $S_*(GK)$, thus, $\Omega S_*(GK)$ is an $\mathcal S_2$ algebra. By unwinding the definitions used in \cite{hess} and interpreting them in our language, the authors construct a natural Hopf twisting morphism $\Omega S_*(K) \to \Omega S_*(GK)$ where $K$ is a $1$-reduced simplicial set, and $GK$ is the Kan loop group,
which induces a weak equivalence $\Omega S_*(K) \to S_*(GK)$. This directly implies that there is a weak equivalence
of Hopf algebras $\Omega S_*(K) \to B\Omega S_*(GK)$. Thus, by interpreting Theorem 4.4 of \cite{hess} in our context, and applying Theorem \ref{unit}, we have the following.
\begin{cor}[Hess--Parent--Scott--Tonks]\label{hessthm}If $K$ is a $1$-reduced simplicial set, then there is a natural equivalence of Hopf Algebras $\Omega S_*(K)\to B\Omega S_*(GK)$. Thus, $\Omega S_*(K)$ and $S_*(GK)$ are naturally equivalent as Hopf algebras.
\end{cor}

Back to generalities, it would be nice to complete the duality picture by showing that the natural map of algebras $g:\Omega B A \to A$
is a map of $\mathcal S_2$ algebras if $A$ is an $\mathcal S_2$ algebra. However, this is not the case in general as can
be seen by considering the higher braces. For example, $[[x]]\{[[y]],[[z]]\} = 0$ in $\Omega B A$ but $g([[a]]) = a$,
and in general $x\{y,z\}$ need not be zero in $A$. Essentially, the functor $\Omega$ is not free enough as an $\mathcal S_2$ algebra. We will address this problem in the following section.

\section{The $\mathcal S_2$ Cobar Construction}\label{tildecobar}
In this section, we freely use both the sequence notation for elements of $\mathcal S_2$ introduced above and the brace notation that we have primarily used so far, for example $(1,2,1)(x,y) = x\{y\}$.

As is clear by now, $\Omega: \HopfAlg \to \SAlg{2}$ is not the left adjoint to $B: \SAlg{2} \to \HopfAlg$; it is simply not free enough as an $\mathcal S_2$ algebra to have the appropriate universal property. We now consider the true left adjoint $\widetilde \Omega : \HopfAlg \to \SAlg{2}$. Our definition of Hopf twisting morphisms \ref{hopftwistdef} makes the construction clear. First, given any Hopf algebra $C$, we can define a differential
$\partial_\Delta$ on $\mathbb S_2 (\Sigma^{-1} \ol C)$ (the free $\mathcal S_2$ algebra on $\Sigma^{-1}\ol C$ defined in Section \ref{conv}) by $\partial_\Delta ([c]) = (-1)^{|\ol c^{(1)}|}[\ol c^{(1)} | \ol c^{(2)}]$ and extending by
the Leibniz rule. It is easy to see that this differential commutes with the natural differential $d$, and thus we have a differential
$\partial = d + \partial_\Delta$.
Define $\widetilde \Omega (C)$ to be the quotient of $(\mathbb S_2(\Sigma^{-1} \ol C), \partial)$ by the ideal $I$ generated by
\[[c_1 c_2] - \sum (-1)^\alpha [c_1]\{[\ol c_2^{(1)}],\ldots, [\ol c_2^{(k)}]\} \]
for all $c_1, c_2 \in \ol C$ and
where $\alpha = k(|c_1|-1) + \sum_{i=1}^k (k-i)(|\ol c_2^{(i)}| - 1)$.

\begin{lem}$\partial I \subset I$
\end{lem}
\begin{proof}It suffices to show that for a generator \[g(c_1,c_2) = [c_1 c_2] - \sum (-1)^\alpha [c_1]\{[\ol c_2^{(1)}],\ldots, [\ol c_2^{(k)}]\},\] we have $\partial g(c_1, c_2) \in I$. This is essentially an exercise in manipulating Identity \ref{DIFF} and computing signs; we omit the details.
\end{proof}

Thus, $\partial$ defines a differential on $\widetilde \Omega (C)$. Therefore, a map $\widetilde \Omega C \to A$ of $\mathcal S_2$ algebras is the same thing as a map $\Sigma^{-1} \ol C \to A$ that
extends to a map $(\mathbb S_2 (\Sigma^{-1} \ol C), \partial) \to A$ such that every generator in $I$ is sent to zero.
This description makes the following proposition clear.
\begin{prop}The functor $\widetilde \Omega : \HopfAlg \to \SAlg{2}$ is left adjoint to $B: \SAlg{2} \to \HopfAlg$.
\end{prop}
The map $C \to \widetilde \Omega C$ is evidently a Hopf twisting morphism, which is in particular a twisting morphism, and we have already seen that the map $C\to \Omega C$ is a Hopf twisting morphism. Thus, we obtain maps $\Omega C \to \widetilde \Omega C \to \Omega C$. The first of these is a map of algebras, and the second is a map of $\mathcal S_2$ algebras. Furthermore, the composition is clearly the identity.

The rest of this section is devoted to proving the following theorem and its corollaries.
\begin{thm}\label{barthm}The $\mathcal S_2$ algebra map $\widetilde \Omega C \to \Omega C$ is a strong associative algebra deformation retraction, with section $\Omega C \to \widetilde \Omega C$.
\end{thm}
\begin{cor}\label{adjcor} The unit map $C \to B \widetilde \Omega C$ and the counit map $\widetilde \Omega B A \to A$ are weak equivalences.
\end{cor}
\begin{proof}If $A\to A'$ is a map of algebras that is a homotopy equivalence of complexes, then using the word length filtration, we see that $BA \to BA'$ is a weak equivalence. Thus, we apply two-out-of-three to $C\to B\widetilde \Omega C \to B\Omega C$ to obtain the first equivalence.
The second follows similarly.
\end{proof}

\begin{cor}\label{equivthm}$\Omega B A$ is weakly equivalent to $A$ as an $\mathcal S_2$ algebra.
\end{cor}
\begin{proof}We have the following zig-zag of $\mathcal S_2$ algebra equivalences: \[\Omega B A \longleftarrow \widetilde \Omega B A \longrightarrow A.\]
\end{proof}
\begin{cor}Assume that $R$ is a field. Then, the adjunction
\[
\xymatrix{\HopfAlg \ar@<1ex>[rr]^{\widetilde \Omega}&& \SAlg{2} \ar@<1ex>[ll]^B\ar@<1ex>[ll]}
\]
is a right semi-model approximation in the sense of \cite{JeromeW}.
\end{cor}
\begin{proof}See Definition $5.1$ of \cite{JeromeW} for the axioms. The category $\HopfAlg$ is a ``category with weak equivalences'' given by the homology isomorphisms. According to Mandell \cite{mandellenfpadic}(Section 13, see also Fresse \cite{fressebook} (Part III Section 12)), $\SAlg{2}$ is a semi-model category. Since $R$ is a field, $\widetilde \Omega $ and $B$ preserve weak equivalences, i.e., they are ``homotopy meaningful''. If $\widetilde \Omega C \to A$ is a weak equivalence, then its adjoint is given by the composite $C \to B\widetilde \Omega C \to BA$. The first map in the composition is a weak equivalence by Corollary \ref{adjcor}. The second map is a weak equivalence because $B$ preserves weak equivalences.
\end{proof}

We will now prove Theorem \ref{barthm}. First, we state a technical lemma that will help to construct the required homotopy.

\begin{lem}\label{freesub}Let $C = TV$ be a tensor Hopf algebra on the free graded $R$-module $V$, where $V$ consists of primitive elements. Then, the composition
\[i: \mathbb S_2(\Sigma^{-1}V) \to \mathbb S_2(\Sigma^{-1}\ol C) \to \widetilde \Omega C\] is injective.
\end{lem}
\begin{proof}The natural degree $-1$ map $V \to \mathbb S_2(\Sigma^{-1}V)$ extends uniquely to a Hopf twisting morphism $f: C \to \mathbb S_2(\Sigma^{-1}V)$ by Corollary \ref{freehopf}. This defines a retraction $r: \widetilde \Omega C \to \mathbb S_2(\Sigma^{-1}V)$ of $i$.
\end{proof}

\begin{proof}[Proof of Theorem \ref{barthm}]
Let $C$ be a Hopf algebra. We will use the notation: $\pi\colon \widetilde \Omega C \to \Omega C$, $\iota\colon \Omega C \to \widetilde \Omega C$, and $p = \iota \pi$. The Hopf algebra embedding $C \to B\Omega C$, together with the bar filtration on $B\Omega C$, induces
a Hopf algebra filtration $F_n C$ on $C$ such that $F_n C = \ker \ol \Delta^{(n+1)}$. This filtration induces an $\mathcal S_2$ algebra filtration on $\widetilde \Omega C$ and $\Omega C$.
We will now define a natural homotopy $h\colon \widetilde \Omega C \to \widetilde \Omega C$, such that
\begin{align}
h^2 &= 0\label{f1}\\
hp &= 0 = ph \label{f2}
\end{align} by induction on the filtration. We consider the case that $C$ has the zero differential. If $C$ has a differential $d_C$, then we can ignore the differential and obtain an $h$ that, once it has been defined, clearly satisfies $hd_C = -d_C h$.

We begin by setting $h(\alpha([c_1], \ldots, [c_k])) = 0$ with $c_i\in \ol C$ primitive, and $\alpha \in \mathcal S_2(k)_0$. This defines a derivation homotopy on filtration level $0$ satisfying (\ref{f1}) and (\ref{f2}).

We next define $h(x)$ for $x = \alpha([c_1], \ldots, [c_k])$, where the $c_i$ are all primitive, by induction on the degree of a single sequence $\alpha \in \mathcal S_2(k)$. By naturality, we can assume that $C = TV$ and $V$ is freely generated by primitives $c_1, \ldots, c_k$. Consider $y = x-p(x)-h\partial x$. Observe that $y \in \mathbb S_2(\Sigma^{-1} V)$, and so by Lemma \ref{freesub}, there is a unique $o_\alpha \in \mathcal S_2(k)$ such that $y = o_\alpha([c_1], \ldots, [c_k]))$. Since $\partial x$ is the sum of elements of the form $\beta([c_1], \ldots, [c_k])$, where $|\beta| < |\alpha|$, we have by induction that $\partial h \partial x = \partial x - p(\partial x)$, therefore $\partial y = 0$, and thus $d o_\alpha = 0$. Our goal now is to find $b_\alpha \in \mathcal S_2(k)$ such that $d b_\alpha = o_\alpha$. Define $j_\alpha$ to be the first repeated integer to occur in the sequence $\alpha$ from left to right, and let $S_\alpha$ be the set of integers that occur before $j_\alpha$.

We now define a general class of operators on $\mathcal S(k)$. Let $1\le j \le k$, and let $S\subset \{1, \ldots, k\}$ such that $j \not\in S$. Then, we define $h_{(j, S)}: \mathcal S(k) \to \mathcal S(k)$ to be given on a sequence by inserting a $j$, after the first integer in the sequence that is not in the set $S$. This is a variation of a homotopy given by McClure--Smith (which they credit to Benson) \cite{McSmith2} (p. 689) to prove that $\mathcal S$ is an $E_\infty$ operad. Thus, we can see that the following equation is satisfied for sequences in which the elements of $S$ appear exactly once:
\[dh_{(j, S)} + h_{(j,S)} d = 1 + t_{(j, S)},\] where $t_{(j, S)}$ is zero if $j$ occurs more than once. If $j$ occurs exactly once and is the first integer in the sequence not in the set $S$, it gives minus the sequence. Otherwise, it gives minus the sequence obtained by deleting the original occurence of $j$ and putting it just before the first integer not in the set $S$.

Then, given $o_\alpha$ we define $b_\alpha = h_{(j_\alpha, S_\alpha)}(o_\alpha)$. Since $o_\alpha$ is a cycle, and the elements of $S_\alpha$ do not repeat in $o_\alpha$, we have that $db_\alpha = o_\alpha + t_{(j_\alpha, S_\alpha)}(o_\alpha)$.

We will now verify two things about the construction above, first that \[t_{(j_\alpha, S_\alpha)}(o_\alpha) = 0,\] and second, that $b_\alpha$ is indeed in $\mathcal S_2(k)$ and not just in $\mathcal S(k)$, as defined. Both will require analyzing the $\mathcal S_2$ operations on $\Omega C$, and we will now proceed by checking special cases. If $j_\alpha$ occurs at least three times in $\alpha$ then $p(x) = 0$, and $j_\alpha$ occurs at least twice in the terms in $h\partial x$. In this case,
$t(o_\alpha) = 0$. Now, we focus on the case when $j_\alpha$ occurs exactly twice. The same will be true for all terms in $\partial x$, except for
those two that come from deleting $j_\alpha$, these terms have opposite signs and cancel after applying $h$ and then $t_{(j_\alpha, S_\alpha)}$. Now we observe that $p(x)=0$ or $p(x)$ has only terms with $j_\alpha$ as a repeat. Thus, in this case $t(o_\alpha) = 0$, as well. If $i$ is an integer appearing in $\alpha$ that does not occur between repeats, then the same is true of $o_\alpha$ and thus $b_\alpha$. Now, in $\alpha$, $j_\alpha$ is not contained between repeats, thus $b_\alpha \in \mathcal S_2$.

We now make some observations about $h$. By induction on the degree of $\alpha$ we can show that $b_\alpha = 0$ or it is a sum of sequences that have at least one integer that occurs at least three times. From this, it follows that $ph = 0$ on filtration level zero, since $\Omega C$ vanishes on such sequences applied to singletons. We now observe from the formulas for the action of $\mathcal S_2$ on singletons in $\Omega C$ that, for the sequences appearing in $p(x)$, every integer occurs at most twice and if $j$ occurs twice then the subsequence between the two occurrences of $j$ is a singleton or of the form $(k,\cdots,k)$, for some $k$. We can again prove by induction on degree that $h(x) = 0$ if $\alpha$ is of this form, and thus $hp = 0$ on filtration level zero.

We can now see by induction on the degree of $\alpha$ that for $z = b_\alpha([c_1], \ldots, [c_k])$, we have $h(z) = 0$, and thus $h^2 = 0$ on filtration level zero. Note that for every sequence $\beta$ in $b_\alpha$, $j_{\beta} = j_\alpha$, we see from above that $p(z) = 0$, and
\begin{align*}h(\partial z) &= h(o_\alpha([c_1], \ldots, [c_k])) \\
&= h(x - p(x) - h\partial x)\\
&= h(x)\\
&= z,
\end{align*}
by induction and the fact that $hp = 0$. Thus, $z - p(z) - h\partial z = 0$, and it follows that $h(z) = 0$, and hence $h^2 = 0$.

 Since the map $p$ is an algebra map, and since we define $h$ from the left, we can show by induction on the degree of $\alpha$ that $h$ is a derivation homotopy from $1$ to $p$ on filtration level zero, that is $h(xy) = h(x)p(y) \pm x h(y)$. We assume that $\alpha$ can be subdivided into two sequences $\alpha_1$, $\alpha_2$ such that $(1,2)(\alpha_1, \alpha_2) = \alpha$, we assume that $x=\alpha_1([c_1], \ldots, [c_k])$ and $y=\alpha_2([c_1], \ldots, [c_k])$, and thus $xy = \alpha([c_1], \ldots, [c_k])$. If $\alpha_1$ contains no repeats, then $h(x) = 0$ and the formula is clear. If $\alpha_1$ contains repeats then $j_\alpha = j_{\alpha_1}$, $S_\alpha = S_{\alpha_1}$, and $t_{(j_\alpha, S_\alpha)}(d\alpha_1) = 0$, thus $h_{(j_\alpha, S_\alpha)}(d\alpha_1) = \alpha_1$. Now, by induction we have the following,
 \[xy - p(xy) - h(\partial(xy)) =\]
  \[ xy - p(x)p(y) - h(\partial (x))p(y) \pm \partial(x) h(y)
 \pm h(x) p(\partial(y)) \pm x p(\partial(y)).
 \]
The terms $- p(x)p(y) - h(\partial (x))p(y)$ give $h(x)p(y)$, the term $\pm \partial(x)h(y)$ gives $\pm xh(y)$, and the rest of the terms will vanish after applying $h_{(j_\alpha, S_\alpha)}$ to the corresponding operation.

Finally, we proceed with the induction by filtration. Assume we have defined the derivation homotopy $h$ satisfying (\ref{f1}) and (\ref{f2}) for all elements in filtration $n$, and let $x = \alpha([c_1], \ldots, [c_k])$ be in filtration level $n+1$. Observe that $\partial x$ is in filtration level $n$, so $h\partial x $ is defined, and let $y = x - p(x) - h\partial x$. It follows from naturality that $y - o_\alpha([c_1], \ldots, [c_k])$ is in filtration level $n$. Thus, we define \[hx = h(y - o_\alpha([c_1], \ldots, [c_k])) + b_\alpha([c_1], \ldots, [c_k]).\] Let $c = ([c_1], \ldots, [c_k])$. By induction, and the fact that $d(b_\alpha) = o_\alpha$ we have
\begin{align*} \partial hx &= \partial h(y - o_\alpha(c)) + \partial b_\alpha (c)\\
&= - h(\partial y \pm o_\alpha(\partial c)) +  y - o_\alpha(c) - p(y) - p(o_\alpha(c)) + o_\alpha(c) \pm b_\alpha(\partial c)\\
&= - h(\pm o_\alpha(\partial c)) +  y  \pm b_\alpha(\partial c)\\
&= x - p(x) - h(\partial x) \pm b_\alpha(\partial c) - h(\pm o_\alpha(\partial c)).
\end{align*}
Now, a consistency check shows that $h(o_\alpha(\partial c)) = b_\alpha(\partial c)$. Thus, we have \[\partial h x + h\partial x = x - p(x).\]
Finally, induction and similar arguments to those given above show that $h$ is a derivation homotopy satisfying (\ref{f1}) and (\ref{f2}).
 \end{proof}

\begin{ex}Let's trace through the first non-trivial example to see how the homotopy works; we work with $R = \Z /2$ to avoid the signs. Consider $x = (1,2,3,1)([c_1], [c_2], [c_3])$, where $c_1, c_2, c_3$ are all primitive, and we let $\alpha = (1,2,3,1)$. Then, we have $h\partial x = 0$, but \[p(x) = [c_1 c_2] [c_3] + [c_2][c_1 c_3] = ((1,2,1,3)+ (2,1,3,1))([c_1], [c_2], [c_3]),\] and thus $x-p(x) = ((1,2,3,1) + (1,2,1,3) + (2,1,3,1))([c_1], [c_2], [c_3])$, and so $o_\alpha = (1,2,3,1) + (1,2,1,3) + (2,1,3,1)$, $j_\alpha = 1$ and $S_\alpha = \emptyset$, and therefore $b_\alpha = (1,2,1,3,1)$. So, finally,
\[hx = (1,2,1,3,1)([c_1], [c_2], [c_3]) = [c_1]\{[c_2], [c_3]\}.\]
\end{ex}

\section{$\Omega (UL)^\vee$ as an $\mathcal S_2$ algebra}\label{chevalley}
In this section, we assume that $1/2 \in R$.

The purpose of this section is to prove that, up to homotopy, $\Omega (UL)^\vee$ is commutative as an $\mathcal S_2$ algebra. The purpose for this is first, to give a nice application of the machinery introduced above, and second, to study $S^*(X)$ as an $\mathcal S_2$ algebra, where $X$ is finite-type. In a forthcoming paper, we prove that under certain assumptions on $X$ that $BS^*(X)$ is equivalent to $UL_X$ as a Hopf algebra. It will then follow from Theorem \ref{barthm}, and this section, that under these assumptions $S^*(X)$ is equivalent as an $\mathcal S_2$ algebra to $C^*(L_X)$ (the Chevalley-Elilenberg cochain complex on the Lie algebra $L_X$), which is a commutative algebra.

Let $L$ be a connected, finite-type, nonnegatively graded differential graded Lie algebra. Then, $(UL)^\vee$ is a Hopf algebra, and thus $\Omega (UL)^\vee$ is an $\mathcal S_2$ algebra by our earlier result. We want to show that this algebra is equivalent as an $\mathcal S_2$ algebra to a commutative algebra. Consider $BUL$, and define $C_*(L) \subset BUL$ to be the largest cocommutative subcoalgebra that contains $\Sigma L$ as its primitives; this is the well-known Chevalley--Eilenberg cochain complex.
It is a classical result of differential homological algebra, see \cite{mooredha} (Section 3) for example, that the inclusion $C_*(L) \to BUL$ is a weak equivalence of coalgebras. Define $C^*(L) = (C_*(L))^\vee$, then $C^*(L)$ is a commutative algebra, and the dual of $C_*(L) \to BUL$ is a weak equivalence of algebras $\Omega (UL)^\vee \to C^*(L)$. Thus, we have a twisting morphism $\alpha: (UL)^\vee \to C^*(L)$ which is the dual of the twisting morphism $C_*(L) \to UL$.

\begin{thm}\label{chev}The twisting morphism $\alpha: (UL)^\vee \to C^*(L)$ is a Hopf twisting morphism.
\end{thm}

Note that from this and Theorem \ref{barthm} we have the following corollary.

\begin{cor}$\Omega (UL)^\vee$ is equivalent as an $\mathcal S_2$ algebra to $C^*(L)$.
\end{cor}

\begin{proof}[Proof of Theorem \ref{chev}]Note that by construction the dual twisting morphism factors $C_*(L) \to L \to UL$.

Let $f, g: \ol{UL} \to R$ be linear maps, corresponding to elements of $\ol{UL^\vee}$. Then, we obtain the product $fg$ by dualizing the coproduct on $UL$, and thus we have that $fg$ is defined on $\ol{UL}$ by
\[
\xymatrix{\ol{UL} \ar[r]^{\ol \Delta}& \ol{UL}\otimes \ol{UL} \ar[r]^{f\otimes g}&  R\otimes R \cong R,}
\]
and $fg$ maps to the element in $C^*(L)$ under $\alpha$ given by restricting the above to $L$. Since $L$ consists of primitive elements, $fg$ maps to zero under the twisting morphism $(UL)^\vee \to C^*(L)$.
\end{proof}

In fact, this calculation can be strengthened to show that $\Omega (UL)^\vee \to C^*(L)$ is actually a map of $\mathcal S_2$ algebras; we will return to this in a future paper.

\bibliographystyle{amsplain}

\begin{thebibliography}{10}

\bibitem{jAdams}
J.~F. Adams, \emph{On the non-existence of elements of {H}opf invariant one},
  Ann. of Math. (2) \textbf{72} (1960), 20--104. \MR{0141119 (25 \#4530)}

\bibitem{anick}
David~J. Anick, \emph{Hopf algebras up to homotopy}, J. Amer. Math. Soc.
  \textbf{2} (1989), no.~3, 417--453. \MR{991015 (90c:16007)}

\bibitem{bergfress}
Clemens Berger and Benoit Fresse, \emph{Combinatorial operad actions on
  cochains}, Math. Proc. Cambridge Philos. Soc. \textbf{137} (2004), no.~1,
  135--174. \MR{2075046 (2005e:18013)}

\bibitem{JeromeW}
Wojciech Chach{\'o}lski and J{\'e}r{\^o}me Scherer, \emph{Homotopy theory of
  diagrams}, Mem. Amer. Math. Soc. \textbf{155} (2002), no.~736, x+90.
  \MR{1879153 (2002k:55026)}

\bibitem{fresse}
Benoit Fresse, \emph{La construction bar d'une alg\`ebre comme alg\`ebre de
  {H}opf {$E$}-infini}, C. R. Math. Acad. Sci. Paris \textbf{337} (2003),
  no.~6, 403--408. \MR{2015084 (2004h:18008)}

\bibitem{fressehopfbar}
\bysame, \emph{{The universal Hopf operads of the bar construction}}, ArXiv
  Mathematics e-prints (2007).

\bibitem{fressebook}
\bysame, \emph{Modules over operads and functors}, Lecture Notes in
  Mathematics, vol. 1967, Springer-Verlag, Berlin, 2009. \MR{2494775
  (2010e:18009)}

\bibitem{GV}
Murray Gerstenhaber and Alexander~A. Voronov, \emph{Homotopy {$G$}-algebras and
  moduli space operad}, Internat. Math. Res. Notices (1995), no.~3, 141--153
  (electronic). \MR{1321701 (96c:18004)}

\bibitem{hess}
Kathryn Hess, Paul-Eug{\`e}ne Parent, Jonathan Scott, and Andrew Tonks, \emph{A
  canonical enriched {A}dams-{H}ilton model for simplicial sets}, Adv. Math.
  \textbf{207} (2006), no.~2, 847--875. \MR{2271989 (2007k:55011)}

\bibitem{HMS}
Dale Husemoller, John~C. Moore, and James Stasheff, \emph{Differential
  homological algebra and homogeneous spaces}, J. Pure Appl. Algebra \textbf{5}
  (1974), 113--185. \MR{0365571 (51 \#1823)}

\bibitem{tKad}
T.~Kadeishvili, \emph{On the cobar construction of a bialgebra}, Homology
  Homotopy Appl. \textbf{7} (2005), no.~2, 109--122. \MR{2156310 (2006b:57048)}

\bibitem{mandellenfpadic}
Michael~A. Mandell, \emph{{$E_\infty$} algebras and {$p$}-adic homotopy
  theory}, Topology \textbf{40} (2001), no.~1, 43--94. \MR{1791268
  (2001m:55025)}

\bibitem{mandell}
\bysame, \emph{Cochains and homotopy type}, Publ. Math. Inst. Hautes \'Etudes
  Sci. (2006), no.~103, 213--246. \MR{2233853 (2007d:55009)}

\bibitem{maygils}
J.~P. May, \emph{The geometry of iterated loop spaces}, Springer-Verlag,
  Berlin, 1972, Lectures Notes in Mathematics, Vol. 271. \MR{0420610 (54
  \#8623b)}

\bibitem{McSmith2}
James~E. McClure and Jeffrey~H. Smith, \emph{Multivariable cochain operations
  and little {$n$}-cubes}, J. Amer. Math. Soc. \textbf{16} (2003), no.~3,
  681--704 (electronic). \MR{1969208 (2004c:55021)}

\bibitem{mooredha}
John~C. Moore, \emph{Differential homological algebra}, Actes du {C}ongr\`es
  {I}nternational des {M}ath\'ematiciens ({N}ice, 1970), {T}ome 1,
  Gauthier-Villars, Paris, 1971, pp.~335--339. \MR{0436178 (55 \#9128)}

\end{thebibliography}
\providecommand{\bysame}{\leavevmode\hbox to3em{\hrulefill}\thinspace}
\providecommand{\MR}{\relax\ifhmode\unskip\space\fi MR }
\providecommand{\MRhref}[2]{%
  \href{http://www.ams.org/mathscinet-getitem?mr=#1}{#2}
}
\providecommand{\href}[2]{#2}

\end{document}